\documentclass[11pt]{article}
\usepackage{amssymb}
\newtheorem{precor}{{\bf Corollary}}

\newenvironment{cor}{\begin{precor}{\hspace{-0.5
               em}{\bf.\ }}}{\end{precor}}
\newtheorem{precon}{{\bf Conjecture}}

\newenvironment{con}{\begin{precon}{\hspace{-0.5
               em}{\bf.\ }}}{\end{precon}}
\newtheorem{predefin}{{\bf Definition}}

\newenvironment{defin}[1]{\begin{predefin}{\hspace{-0.5
                   em}{\bf.\ }}{\rm
#1}\hfill{$\spadesuit$}}{\end{predefin}}
\newtheorem{preexm}{{\bf Example}}

\newtheorem{preappl}{{\bf Application}}

\newtheorem{prelem}{{\bf Lemma}}

\newenvironment{lem}{\begin{prelem}{\hspace{-0.5
               em}{\bf.\ }}}{\end{prelem}}
\newtheorem{preproof}{{\bf Proof.\ }}

\newenvironment{proof}[1]{\begin{preproof}{\rm
               #1}\hfill{$\blacksquare$}}{\end{preproof}}
\newtheorem{presproof}{{\bf Sketch of Proof.\ }}

\newtheorem{prethm}{{\bf Theorem}}

\newenvironment{thm}{\begin{prethm}{\hspace{-0.5
               em}{\bf.\ }}}{\end{prethm}}
\newtheorem{prealphthm}{{\bf Theorem}}

\newenvironment{alphthm}{\begin{prealphthm}{\hspace{-0.5
               em}{\bf.\ }}}{\end{prealphthm}}
\newtheorem{prealphlem}{{\bf  Lemma}}

\newtheorem{prepro}{{\bf Proposition}}

\newtheorem{preprb}{{\bf Problem}}

\def\conct[#1,#2]{\mbox {${#1} \leftrightarrow {#2}$}}
\def\dconct[#1,#2]{\mbox {${#1} \rightarrow {#2}$}}
\def\deg[#1,#2]{\mbox {$d_{_{#1}}(#2)$}}
\def\mindeg[#1]{\mbox {$\delta_{_{#1}}$}}
\def\maxdeg[#1]{\mbox {$\Delta_{_{#1}}$}}
\def\outdeg[#1,#2]{\mbox {$d_{_{#1}}^{^+}(#2)$}}
\def\minoutdeg[#1]{\mbox {$\delta_{_{#1}}^{^+}$}}
\def\maxoutdeg[#1]{\mbox {$\Delta_{_{#1}}^{^+}$}}
\def\indeg[#1,#2]{\mbox {$d_{_{#1}}^{^-}(#2)$}}
\def\minindeg[#1]{\mbox {$\delta_{_{#1}}^{^-}$}}
\def\maxindeg[#1]{\mbox {$\Delta_{_{#1}}^{^-}$}}
\def\isdef{\mbox {$\ \stackrel{\rm def}{=} \ $}}
\def\dre[#1,#2,#3]{\mbox {${\cal E}_{_{#3}}(#1,#2)$}}
\def\pdre[#1,#2,#3]{\mbox {${\cal P}_{_{#3}}(#1,#2)$}}
\def\var[#1,#2]{\mbox {${\rm Var}_{_{#1}}(#2)$}}
\def\ls[#1]{\mbox {$\xi^{^{#1}}$}}
\def\hom[#1,#2]{\mbox {${\rm Hom}({#1},{#2})$}}
\def\onvhom[#1,#2]{\mbox {${\rm Hom^{v}}(#1,#2)$}}
\def\onehom[#1,#2]{\mbox {${\rm Hom^{e}}(#1,#2)$}}
\def\core[#1]{\mbox {$#1^{^{\bullet}}$}}
\def\cay[#1,#2]{\mbox {${\rm Cay}({#1},{#2})$}}
\def\cays[#1,#2]{\mbox {${\rm Cay_{s}}({#1},{#2})$}}
\def\dirc[#1]{\mbox {$\stackrel{\rightarrow}{C}_{_{#1}}$}}
\def\cycl[#1]{\mbox {${\bf Z}_{_{#1}}$}}

\begin{document}
\footnotetext[1]{This research was in part supported by
a grant from research institute for ICT.}
\begin{center}
{\Large \bf On the Biclique cover of the complete graph}\\
\vspace*{0.5cm}
{\bf Farokhlagha Moazami and Nasrin Soltankhah}\\
{\it Department of Mathematics} \\
{\it Alzahra University}\\
{\it Vanak Square 19834 Tehran, I.R. Iran}\\
{\tt f.moazami@alzahra.ac.ir}\\
{\tt soltan@alzahra.ac.ir}\\
\end{center}
\begin{abstract} Let $K$ be a set of $k$ positive integers. A biclique cover of type $K$ of a graph $G$ is a collection of complete bipartite subgraphs of $G$ such that for every edge $e$ of $G$, the number of bicliques need to cover $e$  is a member of $K$. If $K=\{1,2,\ldots, k\}$ then the maximum number of the vertices of a complete graph that admits a biclique cover of type $K$ with $d$ bicliques, $n(k,d)$, is the maximum possible cardinality of a $k$-neighborly family of standard boxes in $\mathbb{R}^d$. In this paper, we obtain an upper bound for $n(k,d)$. Also, we show that the upper bound can be improved in some special cases. Moreover, we show that the existence of the biclique cover of type $K$ of the complete bipartite graph with a perfect matching removed is equivalent to the existence of a cross $K$-intersection family.
\begin{itemize}
\item[]{{\footnotesize {\bf Key words:}\ biclique cover- multilinear polynomial- cross $K$-intersection families.}}
\item[]{ {\footnotesize {\bf Subject classification:} 05C70.}}
\end{itemize}
\end{abstract}
\section{Introduction}
Throughout the paper, we consider only the simple graph. Let $V(G)$ and $E(G)$ denote the vertex set and edge set of the graph $G$. As usual we will use the symbol $[n]$ to denote the set $\{1,2, \ldots, n\}$. By a {\it biclique} we mean the complete bipartite graph. We will denote by $K_{m,m}^-$ the complete bipartite graph with a perfect matching removed. A {\it biclique cover} of a graph $G$ is a collection of bicliques of $G$ such that each edge of $G$ belongs to at least one of the bicliques. In the literature, there are several ways to define a biclique cover problem for different purposes, see \cite{alon, bc1, bollobas2, bllobas1, bc2, haji2, haji, jukna, bcthes, frac}. In this paper we consider the biclique cover of order $k$ and in general of type $K$ that was introduced by Alon~\cite{alon}.
\begin{defin}{Let $K$ be a set of $k$ positive integers. We say that a biclique cover of the graph $G$ is of type $K$ if for every edge $e$ of the graph $G$, the number of bicliques that cover $e$ is an element of the set $K$.
}
\end{defin}
The number of bicliques in a biclique cover is the {\it size} of the cover. We write $n(K,d)$ for the maximum number of vertices of a complete graph  that admits a  biclique cover of type $K$ and size $d$.
In the aforementioned  definition  if we take the set $K=\{1,2,\ldots, k\}$ then this biclique cover is called of {\it order} $k$. Alon~\cite{alon} used the concept of the biclique cover and provided a relation between the biclique cover of order $k$ and size $d$ of the complete graph and $k$-neighborly family of standard boxes in $\mathbb{R}^d$, (see \cite{alon} for more details.) The following result is due to N. Alon~\cite{alon}.
\begin{alphthm}\label{alon}{\em \cite{alon}} Let $d$ be a positive integer and $1\leq k\leq d$ then

  \begin{enumerate}
                \item $d+1=n(1,d)\leq n(2,d)\leq \cdots \leq n(d-1,d)\leq n(d,d)= 2^d.$
                \item ($\frac{d}{k})^k\leq\prod_{i=0}^{k-1}(\lfloor\frac{d+i}{k}\rfloor+1)\leq n(K,d)\leq\sum_{i=0}^{k}2^i{d \choose i} < 2(\frac{2ed}{k})^k.$
              \end{enumerate}
\end{alphthm}
To prove the upper bound in the second relation he constructed multilinear polynomials that are linearly independent. The proof can be easily extended to the biclique cover of type $K$. In this paper, we add some other linearly independent multilinear polynomials and obtain a slightly improvement bound for the mentioned upper bound.
\begin{thm}\label{koli} For $1 \leq k \leq d$, assume that $K$ is a set of $k$ positive integers. Then
$$n(K,d)\leq 2^k{d\choose k}+\sum_{i=1}^{k-1}2^i{d-1 \choose i-1}.$$
\end{thm}
Also, it was shown by Alon~\cite{alon} that for $K=\{2,4,\ldots, 2i\}$ there exists a biclique cover $H_1,H_2,\ldots,H_d$ of type $K$ for the complete graph on $(1+{d \choose 2}+ {d\choose 4}+ \cdots+ {d\choose 2i})$ vertices which improves the lower bound in Theorem~\ref{alon}. To see this, assume that the vertices of the complete graph is denoted by all subsets of cardinality $0,2,4,\ldots,2i$ of $[d]$. Construct the biclique cover $H_1,H_2, \ldots, H_d$ which $H_i$ has $(X_i,Y_i)$ as vertex classes  such that $X_i$ is the set of all subsets that contain $i$ and $Y_i$ is the set of all subsets that do~not contain $i$. These results are far from being optimal, but the above construction motivated us to define the following definition.
\begin{defin}{
Suppose that $\{G_1, \ldots, G_d \}$ is a biclique cover of the graph $G$ where $G_i$ has $(X_i, Y_i)$ as its vertex set. If for every $1\leq i \leq d$, $X_i\cup Y_i= V(G)$, then this biclique cover is called a {\it regular biclique cover}.}
\end{defin}
We will denote by $n_r(K,d)$ the maximum possible cardinality of the vertices of a complete graph such that there exists a regular biclique cover of type $K$ and size $d$. We prove the following upper bound of  $n_r(K,d)$.
\begin{thm}\label{regular} If $K$ is a set of $k$ positive integers, then
$$n_r(K,d)\leq 2^{k-1} {d \choose k}+ 2^{k-1}{d-1 \choose k}+2^k-1.$$
\end{thm}
 It is interesting to consider the biclique cover of type $K$ of other graphs besides the complete graph. Assume that $K$ is a set of $k$ positive integers and $X$ is an arbitrary set of $d$ points. Suppose that ${\cal A} = \{A_1, A_2, \ldots , A_m \}$ and ${\cal B} = \{B_1, B_2, \ldots , B_m \}$ are two collections of subsets of $X$ such that $|A_i \cap B_j| \in K$ for $i \neq j$ and $|A_i \cap B_i| = 0$ for every $i$. The pair $({\cal A}, {\cal B})$ is called a cross $K$-intersection families. The following theorem shows that a cross $K$-intersection family can be
formulated in terms of biclique  cover of type $K$ of the graph $K_{m,m}^-$.
\begin{thm}\label{cross} Let $K = \{l_1, l_2, \ldots,l_k  \}$, there exists a cross $K$-intersection families with $m$ blocks on a set of $d$ points if and only if there exists a biclique cover of type $K$ and size $d$ of $K_{m,m}^-$.
 \end{thm}
In \cite{snevily} Snevily made the following conjecture.
\begin{con}\label{snevily}{\em \cite{snevily}} Let ${\cal A} = \{A_1,A_2, \ldots ,A_m \}$ and ${\cal B} = \{B_1,B_2, \ldots, B_m \}$ be two collections of subsets of an $d$-element set. Let $K = \{l_1, l_2, \ldots , l_k \}$ be a collection
of $k$ positive integers. Assume that for $i \neq j$ we have $|A_i \cap B_j| \in K$ and
that $|A_i \cap B_i|=0$, then $$ m \leq {d \choose k}.$$
\end{con}
 By Theorem~\ref{cross}, we can state the above conjecture in terms of the biclique cover as follows. The maximum number of the vertices in each part of a complete bipartite graph with a perfect matching removed that admits a biclique cover of size $d$ is at most ${d \choose k}$. Note that this bound is sharp by taking all $k$-element subsets of $[d]$ as ${\cal A}$ and all $(d-k)$-element subsets of $[d]$ as ${\cal B}$.
In \cite{Chen} William Y.C. Chen and Jiuqiang Liu  proved the following theorem.
\begin{alphthm}\label{tt}{\em \cite{Chen}}
Let $p$ be a prime number and $K = \{l_1, l_2, \ldots , l_k\} \subseteq \{1, 2, \ldots , p - 1 \}$. Assume that ${\cal A} = \{A_1, A_2, \ldots, A_m \}$ and ${\cal B} = \{B_1, B_2, \ldots, B_m \}$ are two collections of subsets of $X$ such that $|A_i \cap B_j | (mod p) \in K$ for $i \neq j$ and $|A_i \cap B_i| = 0$ for every $i$. If $max l_j < min \{|A_i | (mod p) | 1 \leq i \leq m \}$, then $$ m \leq {d-1 \choose k} + {d-1 \choose k-1} + \cdots + {d-1 \choose k-2r+1},$$
where $r$ is the number of different set sizes in ${\cal A}$.
\end{alphthm}
Clearly, for a prime number $p$ greater than $d$, and $r=1$ the above theorem is true and the following corollary is straightforward.
\begin{cor}\label{corolary}{\em \cite{Chen}} Let $K = \{l_1, l_2, \ldots , l_k\}$ be a set of $k$ positive integers and $max \ l_j < s$. Suppose that ${\cal A} = \{A_1, A_2, \ldots , A_m \}$ and ${\cal B} = \{B_1, B_2, \ldots , B_m \}$ are two collections of subsets of $[d]$ such that $|A_i \cap B_j| \in K$ for $i \neq j$ and $|A_i \cap B_i| = 0$ for every $i$. If either ${\cal A}$ is $s$-uniform or ${\cal B}$ is $s$-uniform, then $$ m \leq {d \choose k}.$$
\end{cor}
This corollary shows that if we have a biclique cover of type $K=\{l_1,\ldots, l_k\}$ of a complete bipartite graph with a perfect matching removed such that every vertex of this graph lies in exactly $s$ bicliques and $max \ l_j < s$. Then the maximum possible cardinality of the vertices of this graph is at most ${d \choose k}$. The structure of the rest of this paper is to prove Theorem~\ref{koli},~\ref{regular}, and \ref{cross}. The proofs based on the concept of applying linear algebra method that is used in~\cite{alon2, Chen, Liu, snevily}.
\section{Proof of Theorem~\ref{koli}}
A polynomial in $n$ variable is called {\it multilinear} if every variable has degree $0$ or $1$. Observe that when each variable in a polynomial attains values $0$ or $1$, if each variable $x_i^p$ $(p>1)$ is replaced by $x_i$, we can consider this polynomial as a multilinear polynomial. For a subset $A_i$ of $[n]$, the {\it characteristic vector} of $A_i$ is the vector $v_{A_i}=(v_1, \ldots, v_n)$, where $v_j=1$ if $j \in A_i$ and $v_j=0$ otherwise.
Let $\{H_1, H_2, \ldots, H_d \}$ be a biclique cover of type $K$ for the graph $K_n$ such that $H_i$ has $X_i$ and $Y_i$ as its vertex classes. For every $1\leq i\leq n$ define
$$ A_i:=\{ j \ | \ i \in X_j \} \,\,\,\ \& \,\,\,\ B_i:=\{ j \ | \ i \in Y_j \}.$$
Now, with each pair $(A_i,B_i)$ we associate a polynomial $P_i(x,y)$ defined by:
$$P_i(x,y)= \prod_{l_j\in K}(v_{A_i}.x + v_{B_i}.y-l_j).$$
 Where $v_{A_i}$ (resp. $v_{B_i}$)  is the characteristic vector of the set $A_i$ (resp. $B_i$), $x=(x_1, x_2, \ldots, x_d)$ and $y=(y_1, y_2, \ldots, y_d )$. The key property of these polynomials is
 \begin{equation}\label{sefr}
P_i(v_{B_i}, v_{A_i})\neq 0 \,\,\ {\rm and} \,\,\  P_i(v_{B_j}, v_{A_j})=0  \,\, {\rm for \ all} \ i \neq j,
\end{equation}
which follows immediately from this fact that $\{H_1, \ldots,H_d\}$ is a biclique cover of type $K$.
Let $${\cal A}=\{(M,N) \ | \ M,N \subseteq [d], \,\ M\cap N= \emptyset, \,\, d\in N, \,\   |M\cup N |\leq k \}.$$  It is easy to see that the cardinality of ${\cal A}$ is equal to $\sum_{i=0}^{k-1}2^{i}{d-1 \choose i}$. For every $(M,N) \in {\cal A}$, the polynomial $Q_{(M,N)}(x,y)$ is defined by $$Q_{(M,N)}(x,y)=\prod_{i\in M}x_i \prod_{i \in N} y_i.$$
Throughout the paper we set $\prod_{i\in A} x_i=\prod_{i\in A} y_i=1$ when $A$ is an empty set.
We will now show  that the polynomials in the set $${\cal P}=\{ P_i(x,y) \ | \ 1\leq i\leq n\} \cup \{ Q_{(M,N)}(x,y) \ | \ (M,N) \in {\cal A} \}$$, as the polynomials from $\{0,1\}^{2d}$ to $\mathbb{R}$, are linearly independent.
For this purpose, we set $$ \sum_{i=1}^n \alpha_i P_i(x,y) + \sum_{(M,N)\in {\cal A}}\beta_{(M,N)} Q_{(M,N)}(x,y)=0.$$
We can rewrite the above equality as follows:
\begin{equation}\label{yek}
 \displaystyle\sum_{d\in A_i} \alpha_i P_i(x,y) + \displaystyle\sum_{d\not\in A_i} \alpha_i P_i(x,y)+ \displaystyle\sum_{(M,N)\in {\cal A}}\beta_{(M,N)} Q_{(M,N)}(x,y)=0.
 \end{equation}
 The proof will be divided into $3$ steps. \\

 {\bf Step 1}.\quad  We begin by proving that for every $i$ which $d\not\in A_i$, it holds that $\alpha_i=0$. In the contrary assume that $i_0$ is a subscript such that $\alpha_{i_0}\neq0$ and $d\not\in A_{i_0}$. Substituting $(v_{B_{i_0}}, v_{A_{i_0}})$ in the equation~\ref{yek}, according to the relation \ref{sefr} and since $d\in N$, all terms in the relation~\ref{yek} but $\alpha_{i_0}P_{i_0}(v_{B_{i_0}}, v_{A_{i_0}})$ vanish. In this way, $\alpha_{i_0}P_{i_0}(v_{B_{i_0}}, v_{A_{i_0}})=0$. Finally, as $P_{i_0}(v_{B_{i_0}}, v_{A_{i_0}})\neq 0$, we have $\alpha_{i_0}=0$ which is a contradiction.\\

{\bf Step 2}.\quad  We will show that for every $i$ which $d\in A_i$, it holds that $\alpha_i=0$. According to the step $1$ we have
\begin{equation}\label{do}
 \displaystyle\sum_{d\in A_i} \alpha_i P_i(x,y) + \displaystyle\sum_{(M,N)\in {\cal A}}\beta_{(M,N)} Q_{(M,N)}(x,y)=0.
 \end{equation}
 Assume that $i_0$ is a subscript such that $\alpha_{i_0}\neq0$. Let $v'_{A_{i_0}}=v_{A_{i_0}}-(0, \ldots, 0, 1)$ and evaluate the equation~\ref{do} in $(v_{B_{i_0}}, v'_{A_{i_0}})$. As $A_i\cap B_i=\emptyset$ we have that $d \not\in B_i$ so for every $i$ in the equation~\ref{do} we have $P_i(v_{B_{i_0}}, v'_{A_{i_0}})=P_i(v_{B_{i_0}}, v_{A_{i_0}})$. From this and since $d\in N$ we conclude that $\alpha_{i_0}P_{i_0}(v_{B_{i_0}}, v_{A_{i_0}})=0$, hence $\alpha_{i_0}=0$.\\

{\bf Step 3}.\quad Note that $Q_{(M',N')}(v_{M'},v_{N'})=1$ and $Q_{(M,N)}(v_{M'},v_{N'})=0$ for any $(M,N)\in{\cal A}$ with $M\cup N \neq M'\cup N'$ and $|M\cup N| \geq |M'\cup N'|$. So all the polynomials of the set $ \{ Q_{(M,N)}(x,y) \ | \ (M,N) \in {\cal A} \}$ are linearly independent. Therefore, for every $(M,N)\in {\cal A}$ we have $\beta_{(M,N)}=0$.

We have thus proved the polynomials of the set ${\cal P}$ as the polynomials with domain $\{0,1\}^{2d}$ are linearly independent. So we can consider these polynomials as multilinear polynomials. On the other hand every  polynomial in the set ${\cal P}$ can be written as a linear combination of the multilinear monomials of degree at most $k$. Furthermore, they do~not contain any monomials that
contain both $x_i$ and $y_i$ for the same $i$. The number of such monomials are $\sum_{i=0}^k 2^i{d \choose i}$ and hence,
$$n+ \sum_{i=0}^{k-1} 2^{i}{d-1 \choose i}\leq \sum_{i=0}^{k} 2^i{d \choose i}.$$
Now, by a straightforward calculation the formula of Theorem~\ref{koli} will be achieved.
\section{Proof of Theorem~\ref{regular}}
Before embarking on the proof of Theorem~\ref{regular}, we will establish the following lemma.

\begin{lem} \label{poly}For every $1\leq i \leq d-1$, let the set ${\cal B}_i$ define as follows: $${\cal B}_i=\{(I,J)| I,J \subseteq [d-i+1], \ I\cap J= \emptyset, \  d-i+1\in I\cup J, \, |I\cup J|\neq d-i+1,  \  |I\cup J |\leq k-1 \}.$$
Let for every pair $(I,J)\in {\cal B}_i$ ,$R_{(I,J)}^i(x,y)$ denote the following polynomial
$$R_{(I,J)}^i(x,y)=\prod_{j\in I}x_j\prod_{j\in J}y_j(\sum_{j\not\in J, j\leq d-i}x_j+\sum_{j\not\in I, j\leq d-i}y_j-(d-i)).$$
Then
\begin{equation}\label{set}
{\cal B}=\{R_{(I,J)}^i(x,y) \ | \ 1\leq i\leq d-1, \,\  (I,J)\in {\cal B}_i \}
 \end{equation}
 is a set of linearly independent polynomials.
\end{lem}
\begin{proof}{ To prove the assertion, assume this is false and let
\begin{equation}\label{lem}
\sum_{(I,J)\in {\cal B}_1} \gamma_{(I,J)}^1R_{(I,J)}^1(x,y)+ \cdots + \sum_{(I,J)\in {\cal B}_{d-1}} \gamma_{(I,J)}^{d-1}R_{(I,J)}^{d-1}(x,y)=0
\end{equation}
be a nontrivial linear relation. Suppose that  $i_0$ is the greatest superscript and $(I_0, J_0)$ is the subscript such that has minimum cardinality in the set ${\cal B}_{i_0}$ and $\gamma_{(I_0,J_0)}^{i_0}\neq 0$. Substitute $(v_{I_0},v_{J_0})$ for $(x,y)$ in the linear relation~\ref{lem}. In view of the definition of ${\cal B}_{i}$ all terms in the linear relation~\ref{lem} but $\gamma_{(I_0,J_0)}^{i_0}R_{(I_0,J_0)}^{i_0}(v_{I_0},v_{J_0})$ vanish. Since $R_{(I_0,J_0)}^{i_0}(v_{I_0},v_{J_0})\neq 0$, we have  $\gamma_{(I_0,J_0)}^{i_0}=0$. This is a contradiction which completes the proof.
}
\end{proof}
Obviously, for every $1\leq i < d$ we have
$$ |{\cal B}_i|=\left \{\begin{array}{ll}
             \sum_{j=0}^{k-2}2^{j+1}{d-i \choose j} &  \,\,\,\ i \leq d-k+1 \\
                                                    &                     \\
             \sum_{j=0}^{d-i-1}2^{j+1}{d-i \choose j} &  \,\,\,\  i\geq d-k+2
            \end{array} \right.
.$$
 It is a well-known fact that $$\sum_{i=0}^n {m+i\choose m}={m+n+1\choose m+1}.$$
By this fact clearly,  $|{\cal B}|=\sum_{j=1}^{k-1}2^j{d \choose j}-2^k+2$. Let ${\cal A}_k=\{(M,N) \ | \ (M,N)\in {\cal A}, \  |M\cup N|=k \}$  and ${\cal B}$ is defined as Lemma~\ref{poly}. We claim that $ \{ P_i(x,y) \ | \ 1\leq i\leq n\}\cup \{Q_{(M,N)} \ | \ (M,N)\in {\cal A}_k \}$  with all the polynomials $R_{(I,J)}^i(x,y)\in {\cal B}$ remain linearly independent. Before prove the claim, we shall note that all polynomials in the set ${\cal B}$  have this property that vanish in the point $(v_{B_i},v_{A_i})$ for every $1\leq i\leq n$. Now, assume the claim  is false and let
\begin{equation}\label{se}
 \sum_{i=1}^n \alpha_i P_i(x,y) + \sum_{(M,N)\in {\cal A}_k}\beta_{(M,N)} Q_{(M,N)}(x,y)+ \sum_{i=1}^{d-1}\sum_{(I,J)\in{\cal B}_i}\gamma_{(I,J)}^i R_{(I,J)}^i(x,y)=0
 \end{equation}
 be a nontrivial linear relation.\\

{\bf Step 1}.\quad Let $i_0$ be a subscript such that $d\not\in A_{i_0}$ and $\alpha_{i_0}\neq 0$. Substitute $(v_{B_{i_0}},v_{A_{i_0}})$ for $(x,y)$ in the linear relation~\ref{se}. We know that for every $i$ and every $(I,J)\in {\cal B}_i$, $R_{(I,J)}^i(v_{B_{i_0}},v_{A_{i_0}})=0$. Also, $d\in N$ so $Q_{(M,N)}(v_{B_{i_0}},v_{A_{i_0}})=0$. Using these and by \ref{sefr} all terms in the linear relation~\ref{se} but $\alpha_{i_0} P_{i_0}(v_{B_{i_0}},v_{A_{i_0}})$ vanish. Since $P_{i_0}(v_{B_{i_0}},v_{A_{i_0}})\neq 0$, therefore $\alpha_{i_0}=0 $.\\

{\bf Step 2}.\quad According to the step $1$ we have
\begin{equation}\label{chahar}
 \displaystyle\sum_{d\in A_i}^n \alpha_i P_i(x,y)+\sum_{(M,N)\in {\cal A}_k}\beta_{(M,N)} Q_{(M,N)}(x,y)+ \sum_{i=1}^{d-1} \sum_{(I,J)\in{\cal B}_i}\gamma_{(I,J)}^iR_{(I,J)}^i(x,y)=0.
 \end{equation}

Let $i_0$ be a subscript such that $\alpha_{i_0}\neq 0$. We define $v'_{A_{i_0}}$ to be $v_{A_{i_0}}-(0,\ldots, 0,1)$. Substitute $(v_{B_{i_0}},v'_{A_{i_0}})$ for $(x,y)$ in the linear relation~\ref{chahar}. For $1\leq i\leq d-1$ and every pair $(I,J)\in {\cal B}_i$, by the definition of $R_{(I,J)}^i(x,y)$, it holds that  $R_{(I,J)}^i(v_{B_j},v_{A_j})=R_{(I,J)}^i(v_{B_j},v'_{A_j})$ for every $1\leq j \leq n$. Now, similar in the step $1$ all terms in the relation~\ref{chahar} but $\alpha_{i_0} P_{i_0}(v_{B_{i_0}},v'_{A_{i_0}})$ vanish. Since $P_{i_0}(v_{B_{i_0}},v'_{A_{i_0}})\neq 0$, therefore $\alpha_{i_0}=0$. So, we have
 \begin{equation}\label{panj}
 \sum_{(M,N)\in {\cal A}_k}\beta_{(M,N)} Q_{(M,N)}(x,y)+ \sum_{i=1}^{d-1}\sum_{(I,J)\in{\cal B}_i}\gamma_{(I,J)}^iR_{(I,J)}^i(x,y)=0.
 \end{equation}\\

{\bf Step 3}.\quad Since $d\in N$ and $ d\not\in J$ for all $(I,J)\in {\cal B}_i$, $i=2, \ldots, d-1$, if we evaluate equality~\ref{panj} in $(v_I,v_J)$ then we conclude that $\gamma_{(I,J)}=0$. Hence we have
\begin{equation}\label{shish}
 \sum_{(M,N)\in {\cal A}_k}\beta_{(M,N)} Q_{(M,N)}(x,y)+\sum_{(I,J)\in{\cal B}_1}\gamma_{(I,J)}^1R_{(I,J)}^1(x,y)=0.
 \end{equation}
Assume that $(I_0,J_0)$ is a subscript such that has minimum cardinality in the set ${\cal A}_1$ such that $\gamma_{(I_0,J_0)}\neq 0$. Substituting $(v_{I_0},v_{J_0})$ in the equation~\ref{shish} then since $|I_0\cup J_0|\leq k-1$ and $|M\cup N|=k$ all terms in \ref{shish} but $\gamma_{(I_0,J_0)}R_{(I_0,J_0)}^1(v_{I_0},v_{J_0})$ vanish. So we have $\gamma_{(I_0,J_0)}=0$
\\
{\bf Step 4}.\quad  By  independence of the polynomials in the set ${\cal P}$, each $\beta_{(M,N)}=0$, therefore the claim is true. \\

So, we have $n+2^{k-1}{d-1 \choose k-1}+\sum_{i=1}^{k-1}2^i{d\choose i}-2^k+2$ linearly independent polynomials which, as the proof of Theorem~\ref{koli}, are in the space generated by $\sum_{i=0}^{k}2^i{d \choose i}$ monomials. Hence,
 $$n(K,d)\leq 2^k {d \choose k}- 2^{k-1}{d-1 \choose k-1}+2^k-1.$$ That complete the proof of Theorem~\ref{regular}.

\section{Proof of Theorem~\ref{cross}}
 Let $K = \{l_1, l_2, \ldots,l_k  \}$, and $({\cal A},{\cal B})$ be a cross $K$-intersection families with $m$ blocks on a set of $d$ points. For every $1\leq j\leq d$, let
$$X_j\isdef \{i|\  1\leq i \leq m, j\in A_i\},$$

$$Y_j\isdef \{i|\  1\leq i \leq m, j \in B_i\}.$$
 Now, for $j = 1,2, \ldots, d$, we construct  the complete bipartite graph $G_j$ with vertex set $(X_j,Y_j)$, where $X_j$ and $Y_j$  were defined as above. Let $ij$ be an arbitrary edge of $K_{m,m}^-$, consider sets $A_i$ and $B_j$. Without loss of generality assume that $A_i\cap B_j = \{v_1, v_2, \ldots, v_l \}$, which $1\leq l \leq k$. It is not difficult to see that the edge $ij$ was covered by the graphs $G_{v_1}, G_{v_2}, \ldots, G_{v_l}$. So we have a biclique cover of type $K$ and size $d$ of $K_{m,m}^-$. Conversely let $\{G_1, G_2, \ldots, G_d \}$ be a biclique cover of type $K$ and size $d$ of the graph $K_{m,m}^-$. Assume $G_i$  has $(X_i,Y_i)$ as the vertex set. For every $1\leq j\leq m$ define
 $$A_j\isdef \{i|\  1\leq i \leq d, j\in X_i\},$$

$$B_j\isdef \{i|\  1\leq i \leq d, j \in Y_i\}.$$
 Let ${\cal A}=\{ A_1, \ldots, A_m\}$ and ${\cal B}=\{ B_1, \ldots, B_m\}$. Since $X_i \cap Y_i = \varnothing$ for every $1\leq i\leq d$ then $A_j\cap B_j = \varnothing$ for every $1\leq j\leq m$. Also, if $\{ G_{v_1}, G_{v_2}, \ldots, G_{v_l} \}$ is the set of graphs that cover the edge $ij$ then $A_i\cap B_j = \{ v_1, v_2, \ldots, v_l \}$ where $|A_i\cap B_j| \in K$. Hence $({\cal A}, {\cal B})$ is a cross $K$-intersection family.

 {\bf Acknowledgment}\\ The authors wish  to express their tanks to Professor
Hossein Hajiabolhassan for the useful conversations during the preparation of the paper.



\def\cprime{$'$}

\end{document}